\documentclass[12pt]{article}
\usepackage{amssymb, amsmath, amsthm}
\usepackage{graphicx}

\theoremstyle{plain}
\newtheorem{lemma}{Lemma}
\newtheorem{theorem}{Theorem}
\theoremstyle{definition}
\newtheorem{remark}{Remark}

\newtheorem{definition}{Definition}
\newtheorem{example}{Example}

\renewcommand\L{\mathcal{L}}
\newcommand\R{\mathbb{R}}
\newcommand\M{\mathbb{M}}

\newcommand\bfe{\ensuremath{\mathbf{e}}}

\newcommand\Range[1]{{\ensuremath\mbox{Range}(#1)}}

\begin{document}

\title{Numerical approximation of conditionally invariant measures via Maximum Entropy}
\markboth{Bose/Murray}{MAXENT for ACCIMs}
\author{Christopher Bose\footnote{University of Victoria, Mathematics and Statistics, PO BOX 3060 STN CSC, Victoria, BC, Canada V8W3R4.  Email: {\sf cbose@uvic.ca}} ~and
Rua Murray\footnote{University of Canterbury, Department of Mathematics
and Statistics, Private Bag 4800, Christchurch 8140, New Zealand. Email: {\sf rua.murray@canterbury.ac.nz}}}
 
\maketitle

\begin{abstract}
It is well known that open dynamical systems can admit an uncountable
number of (absolutely continuous) conditionally invariant measures (ACCIMs)
for each prescribed escape rate. We propose and illustrate a
convex optimisation based selection scheme (essentially maximum entropy) for
gaining numerical access to some of these measures. The work is similar
to the Maximum Entropy (MAXENT) approach for calculating absolutely continuous invariant 
measures of nonsingular dynamical systems, but contains some 
interesting new twists, including: (i) the natural
escape rate is not known in advance, which can destroy convex structure
in the problem; (ii) exploitation of convex duality to solve
each approximation step induces important (but dynamically relevant and
not at first apparent) localisation of support; (iii) significant potential
for application to the approximation of other dynamically
interesting objects (for example, invariant manifolds).
\end{abstract}

\section{Introduction}

Classical dynamical systems concerns the existence and stability
of invariant sets under the action of a transformation $T:X\to X$. Depending
on the setting, $X$ may be a measure space, a topological space (with or without
a metric structure), a differentiable manifold, a Banach space, and so on.
In each case, orbits defined by iterative application of~$T$ remain
in $X$. For an {\bf open dynamical system\/}, $T$ is defined only on a subset
$A\subsetneq X$, and there are $x\in A$ for which $T(x)\notin A$. Such $x$ 
are said to {\em escape\/}.

Open dynamical systems may be studied in their own right (the paper of 
Demers and Young~\cite{DY06} gives a summary of important questions),
or may be used to study metastable states in closed dynamical systems.
In the latter case, a subset $A\subset X$ is {\em metastable\/} if 
$T(A)\setminus A$ is in some sense {\em small relative to $A$\/}.
Work on making this precise dates at least to 1979, when 
Pianigiani~\&~Yorke~\cite{PY79} introduced {\bf conditionally invariant 
measures\/} (see Section~\ref{sec:2} below) and used them to study 
metastability in expanding interval maps\footnote{The motivation 
in~\cite[p353]{PY79} went beyond interval maps, including preturbulent 
phenomena in the now famous Lorenz equations, and metastable structures in 
atmospheric and other fluid flows and complex systems.}. More recently, 
Homburg~and~Young~\cite{HY02} made productive 
use of conditionally invariant measures to analyse intermittent behaviour 
near saddle-node and boundary crisis bifurcations in unimodal families. 
Many authors have continued to obtain results connecting escape rates and 
metastable behaviour of closed systems; see, for example, 
\cite{AB10,BV12,DW12,FMS11,GTHW11,KL09}.

One of the interesting challenges is to find 
conditionally invariant measures which model the escape statistics 
of orbits distibuted according to some ``natural'' initial measure~$m$
on $A$. In closed dynamical systems there may exist a unique ergodic 
invariant measure~$\mu$ which is absolutely continuous (AC) with respect 
to $m$. Via Birkhoff's ergodic theorem, such $\mu$ describe the orbit
distibution of large\footnote{In the sense of positive $m$-measure.} sets
of initial conditions. By contrast, an open system may support uncountably
many AC conditionally invariant measures (ACCIMs)~\cite[Theorem~3.1]{DY06}, 
so ascribing dynamical 
significance on the basis of absolutely continuity alone does not make
sense. Recently, progress has been made in a variety of settings,
identifying ACCIMs whose densities arise as eigenfunctions of certain
quasicompact {\em conditional transfer operators\/} acting on suitable 
Banach spaces. Such ACCIMs may be considered ``natural'' (see~\cite{DY06} 
for discussion), giving a well-defined {\em escape rate\/} from $A$.  
See, for example,  
\cite{BDM10} for dynamics on Markov towers; \cite{Dem05a,Dem05b} for
interval maps modelled by Young towers; \cite{CMS94,CMS97} for 
expanding circle maps and subshifts of finite type; 
\cite{LiM03} for interval maps with BV potentials. Extending these 
techniques to higher-dimensional settings such as billiards and 
Lorentz gas is an area of much current
interest~\cite{demers13}.

This chapter develops a new class of
computational methods for the explicit
approximation of conditionally invariant 
probability measures on~$A$. Our ideas use
convex optimisation:  the criteria for conditional
invariance are expressed as a sequence of {\em moment
conditions\/} over $L^1$ (integration against a suitable
set of $L^\infty$ test functions), and the principle of maximum
entropy (MAXENT) is used to select (convergent) sequences
of {\em approximately conditionally invariant measures\/}.
The entropy maximisation is solved via standard convex duality
techniques, although attainment in the dual problem may
necessitate a non-obvious (but dynamically meaningful) reduction 
of the domain on which the maximisation is done. The required 
steps are achievable for piecewise constant test functions
(similar in spirit to Ulam's method~\cite{FD03} but with a completely
different mathematical foundation). The chapter is structured as
follows: first, we introduce notation for our study of open systems
and formulate the ACCIM problem (and its uncountable multiplicity of 
solutions) via {\em conditional transfer operators\/}; next, the MAXENT 
problem is set up and analysed; the Ulam-style test functions are 
introduced in Section~\ref{sec:three}, and the domain reduction and 
some numerical examples are given to illustrate the method; we finish 
with some concluding remarks.

\subsection{Nonsingular open dynamical systems}
\label{sec:1}

Let $(X,m)$ be a measure space. 
We consider the dynamics generated by a transformation on a 
{\bf subset of $X$ which fails to be forward invariant\/}; 
such a dynamical system is called {\bf open\/} and may or 
may not support any recurrent behaviour. Let $A\subsetneq X$ be 
measurable and let $T:A\rightarrow X$ be a measurable
transformation where
\begin{itemize}
\item $H_0:=T(A)\setminus A$ is a measurable subset of $X$ (called the {\em hole\/}); and
\item $m(A\cap T^{-1}H_0)>0$; and
\item $m(E)>0$ whenever $m(T^{-1}E)>0$ and $E$ is a measurable subset of $X$; and
\item $T$ is locally finite-to-one (for each $x\in A$, 
$T^{-1}(x)=\{x_{-1}\in A~:~T(x_{-1})=x\}$ is either empty or finite).
\end{itemize}
 
\begin{definition} Let $m|_A$ denote the restriction of the measure~$m$ to $A$.
We call\footnote{Clearly $m\circ T^{-1}\ll m$ so that 
$T:(A,m|_A)\to (X,m)$ is a nonsingular transformation,
but $T:(A,m|_A)\to (X,m|_A)$ fails to be non-singular, as
$m|_A\circ T^{-1}(H_0) = m(A\cap T^{-1}(H_0))>0$ while $m|_A(H_0)=0$.} 
  $(T,A,m|_A)$ satisfying the above conditions a {
\bf nonsingular open dynamical system\/}.
\end{definition}

Notice that $T(x)$ is defined only for $x\in A$, and the ``hole''~$H_0$
can be used to define a {\em survival time\/} for each $x\in A$:
$$\tau(x) := \left\{ \begin{array}{ll}
n &\mbox{if~} x, T(x),\ldots T^n(x)\in A~\mbox{and}~T^{n+1}(x)\in H_0 \\
\infty & \mbox{if~}T^k(x)\in A\;\forall k\in \mathbb{Z}_+. \end{array}\right.$$
When $\tau(x)=n<\infty$, 
$T^n(x)\in H_1:=A\cap T^{-1}(H_0)$  
and such orbits of $T$ terminate at time $\tau(x)+1$.
Only those $x$ for which $\tau(x)=\infty$ can exhibit recurrent behaviour. 

For all that follows it is convenient to decompose $A$ into invariant and
transient parts. Define:
\begin{itemize}
\item the {\bf $n$~step survivor set\/} as
$$A_n := \{x\in A~:~\tau(x)\geq n\} =
\left\{x~:~x,T(x),\ldots,T^n(x)\in A\right\}=\cap_{k=0}^n T^{-k}A.$$
\item $A_\infty := \cap_{n\geq 0} A_n$
\item $H_n:= A_{n-1}\setminus A_{n}=\{x~:~\tau(x)=n-1\}$ for $1<n<\infty$
\end{itemize}
Notice that if $x\in H_n$ then $T^k(x)\in H_{n-k}$ for $0<k\leq n$. The
orbit of $x$ ``falls into the hole" at time~$n$ (escapes) and is lost to
the system thereafter. As well as escape from $A$, we need to account
for the possibility that backwards orbits may not be defined
($T:A_1\rightarrow A$ may not be {\em onto\/}). Since 
some $x\in A$ may have no preimages in~$A$, 
define the following subsets of~$A$:
\begin{itemize}
\item $K_0:=\{x~:~A \cap T^{-1}x =\emptyset\}$
\item $K_n:=\{x~:~\emptyset\neq(A\cap T^{-n}(x))\subset K_0\}
=\{x~:~\min\{k~:~A\cap T^{-k}x=\emptyset\}=n+1\}$
\item $K_\infty:=\{x_0~:~\mbox{there is no sequence $\{x_{-n}\}_{n=1}^\infty$ 
such that $T(x_{-n})=x_{-(n-1)}$, $n>0$}\}$
\item $H_\infty := \cup_{n>0}(H_n\setminus K_\infty)$
\end{itemize}

Points in $K_\infty$ are `backward transient', while points in 
$H_\infty$ are `forward transient'.
Lemma~\ref{lem:lemma1} contains some facts about the action of $T$ on
the various sets $H_n,K_n$. The reader may easily verify that
\begin{itemize}
\item $A_0=A$ and $T(A_n)\subseteq A_{n-1}$
\item $H_n\cap H_m=\emptyset$ if $n\neq m$, $H_n\subseteq A_{n-1}$ 
and $H_n\cap A_n=\emptyset$
\item $T(H_n)\subseteq H_{n-1}$
\item $A\cap T^{-1}(K_{n})\subseteq \cup_{m<n}K_m$  and
$K_{n+1}\subseteq T(K_n)$
\item  $\cup_{n=0}^\infty K_n\subseteq K_\infty$, and the 
union on the left may be finite or infinite (or even 
the emptyset if $T$ is onto~$A$)
\end{itemize}
Any of the containments above may be strict.
In order to avoid unduly messy formulas, from this point on we will generally assume the range of the map $T^{-1}$ is restricted to $A$.

\renewcommand\theenumi{\alph{enumi}}

\begin{lemma}\label{lem:lemma1} 
Let $(T,A,m|_A)$ be a nonsingular open dynamical system.
If $\Omega := A_\infty\setminus K_\infty$ then $A$ admits the
disjoint decomposition $A=K_\infty\cup \Omega \cup H_\infty$ and
\begin{enumerate}
\item $T^{-1}(\cup_{n\geq 0}K_n)\subseteq \cup_{n\geq 0} K_n \pmod{m|_A}$;
\item $T(\Omega) = \Omega$;
\item $T:(H_n\setminus K_\infty)\to(H_{n-1}\setminus K_\infty)$ is onto and
nonsingular (with respect to the obvious restrictions of $m$);
\item $K_\infty=\cup_{n=0}^\infty K_n$.
\end{enumerate}
\end{lemma} 

\begin{proof} (a) Note that
$T^{-1}K_{n}\subseteq \cup_{m< n}K_m$ (for each $n> 0$)
and $T^{-1}K_0=\emptyset$.\newline 
(b) If 
$x\in\Omega$ then $x\in A_\infty$ so $T^n(x)\in A_\infty$.
Thus $\Omega$ is the set of points whose future orbit is contained
in $A$ and has at least one backwards orbit in $A$.\newline  
(c) Let $x\in H_{n-1}\setminus K_\infty$. 
Then there is a sequence $\{x_{-k}\}_{k=1}^\infty$ 
such that $T(x_{-k})=x_{-(k-1)}$ and $T(x_{-1})=x$. 
Clearly $x_{-1}\in H_n\setminus K_\infty$. 
\newline
(d) First, suppose that
$x\notin \cup_{n\geq 0}K_n$. Then $x\notin K_0$ so $\emptyset\neq T^{-1}x$. 
If $T^{-1}x\subseteq \cup_{n\geq 0} K_n$ then there are 
$N_1,\ldots, N_j$ such that 
$T^{-1}x\subseteq K_{N_1}\cup\cdots \cup K_{N_j}$. 
Putting 
$N=1+\max\{N_1,\ldots, N_j\}$ one has $x\in K_N$, a contradiction.
Thus, there is at least one $x_{-1}\in T^{-1}x$ such that 
$x_{-1}\notin \cup_{n\geq 0}K_n$. The proof is completed by induction.
\end{proof}

\begin{example}\label{eg:1}
Let $X=\R^2$, $A=[0,1]^2$ and $T(x,y)=(2x,1/2y)$. 
Then $H_n=(2^{-n},2^{-(n-1)}]\times[0,1]$,
$A_\infty=\{0\}\times[0,1]$. 
On the other hand, $K_n=[0,1]\times(2^{-(n+1)},2^{-n}]$, so
$K_\infty=[0,1]\times(0,1]$. The ``recurrent set'' 
$A_{\infty}\setminus K_\infty=\{(0,0)\}$ is a fixed point 
(so genuinely recurrent), 
and $A_\infty\cap K_{\infty}=\{0\}\times(0,1]$  is part of the stable 
manifold to $(0,0)$. Notice that $H_\infty=(0,1]\times\{0\}$ 
is part of the unstable manifold to $(0,0)$.
\end{example}


\subsection{Escape, conditionally invariant measures and their supports}
\label{sec:2}

We now make precise the notion of escape rates and establish some important
connections with conditionally invariant measures.

\begin{definition} 
The {\bf escape rate\/} of a probability measure~$m_0$ on $A$ is 
$$\rho_{m_0}:=\lim_{n\to\infty} -\frac{1}{n}\log m_0(A_n)
=\lim_{n\to\infty} -\frac{1}{n}\log m_0\{x~:~\tau(x)\geq n\}$$
(when such a limit exists). The open
system $(T,A,m|_A)$ will satisfy the {\bf escape hypothesis\/} iff
\begin{equation}\label{eqn:escape}
m(A_\infty)=0.
\end{equation}
\end{definition}

Clearly, if there is an escape rate $\rho_m>0$ then~(\ref{eqn:escape}) holds. 

\begin{definition}
A probability measure $\mu$ on $A$ is a {\bf conditionally invariant measure\/} (CIM) iff there is $\alpha\in(0,1)$ such that 
$$\mu(T^{-1}E) = \alpha\,\mu(E)\qquad\forall\mbox{measurable~} E\subseteq A.$$
\end{definition}

Note that if $\mu$ is a CIM then
$$ \mu\{\tau\geq n\} = \mu(A_n) = \mu(A\cap T^{-1}A_{n-1}) 
= \alpha\,\mu(A_{n-1}) = \cdots =  \alpha^n\,\mu(A) = \alpha^n.$$
Thus $\rho_\mu=-\log\alpha$ and
$\mu\{x~:~\tau(x)\geq n\} = \mu(A_n) = e^{-\rho_\mu\,n}$, so that
initial conditions distributed according to $\mu$ display 
geometric escape. Provided $H_\infty\neq \emptyset$, 
Lemma~\ref{lem:lemma1}(c) implies the
existence of at least one backwards semi-orbit~$\{x_{-k}\}_{k\geq 0}$ (with 
$T(x_{-k})=x_{-(k-1)}$). Demers and Young~\cite{DY06} point out that
a CIM can be obtained as $(1-\alpha)\sum_{k=0}^\infty\alpha^k\delta_{x_{-k}}$.
However, such CIMs describe only a single orbit, and it remains an 
interesting challenge to find conditionally invariant measures which model 
the escape statistics of the ``natural'' initial measure~$m|_A$.

The domain decomposition of Lemma~\ref{lem:lemma1} and the 
following Lemma~\ref{lem:decomposition} reveal that 
that $A$ decomposes into three pieces:
\begin{enumerate}
\item[(i)] a backwards transient part $K_\infty$ which cannot support any
CIMs, but includes any local basins of attraction 
(we will later identify numerically certain parts of $K_\infty$ and exclude them
for computational reasons). The intuition behind this fact is that the lack
of preimages of points in $K_\infty$ means there is no way to ``replenish'' mass
which is lost to the hole;\newline
\item[(ii)] an envelope 
$\Omega=A_\infty\setminus K_\infty$ for the ``recurrent'' piece 
which can support invariant measures, but not CIMs; and\newline
\item[(iii)] a transient part $H_\infty$ which is the place
to look for CIMs (and includes any local unstable manifolds).
\end{enumerate}


\begin{lemma} \label{lem:decomposition}
Let $(T,A,m|_A)$ be a nonsingular open dynamical system and let
$\Omega$, $K_\infty$ and  $H_\infty$ be as defined previously. Then
\begin{enumerate}
\item if $\mu$ is an invariant or conditionally invariant measure 
on $A$ then $\mu(K_n)=0$ for all $n$ (and $\mu(K_\infty)=0$);
\item if $\mu$ is an invariant measure then $\mu(H_\infty)=0$;
\item if $\mu$ is a conditionally invariant measure then $\mu(\Omega)=0$.
\end{enumerate}
\end{lemma}

\begin{proof} (a) Suppose that $\mu\circ T^{-1}= \alpha\,\mu$ for some 
$\alpha\in(0,1]$. Then
$$\alpha^{n+1}\,\mu(K_n) 
= \mu\circ T^{-(n+1)}(K_n)= \mu\circ T^{-1}(T^{-n}K_n)
\leq \mu(T^{-1}K_0) = \mu(\emptyset)=0.$$
By part~(d) of Lemma~\ref{lem:lemma1}, $\mu(K_\infty)
=\mu(\cup_n K_n)\leq \sum_n\mu(K_n)=0$.\newline 
(b) If $\mu$ is an invariant measure and $\mu(H_n)>0$ then by the 
Poincar\'e recurrence theorem almost every $x\in H_n$ recurs to $H_n$ 
infinitely often. But if $x\in H_n$ then $\{k>n~:~T^kx\in H_n\}=\emptyset$, 
so $\mu(H_n)=0$. It follows that $\mu(\cup H_n)=0$ and hence $\mu(H_\infty)=0$. 
\newline
(c) By Lemma~\ref{lem:lemma1}~(b), 
$\Omega\subseteq T^{-1}(T(\Omega))\subseteq T^{-1}\Omega$ so that
$$\mu(\Omega) \leq \mu\circ T^{-1}(\Omega) 
=\alpha\,\mu(\Omega)<1\,\mu(\Omega).$$
Hence $\mu(\Omega)=0$.
\end{proof}

\noindent {\bf Example~\ref{eg:1} revisited.}
Let $X=\R^2$, $A=[0,1]^2$ and $T(x,y)=(2x,1/2y)$. 
Since $\Omega=\{0\}$,
$K_\infty=[0,1]\times(0,1]$
and $H_\infty=(0,1]\times \{0\}$, the only invariant measure is concentrated on 
the fixed point at $0$ and all CIMs are concentrated on $H_\infty$ (the
unstable manifold to $(0,0)$).

\begin{remark} As suggested already, a discrete variant of the 
set $K_\infty$ arises naturally in the numerical methods 
described below. When $T$ is countable-to-one, it can 
occur that $K_\infty\neq \cup_n K_n=:K'_\infty$, but this 
does not alter the result of Lemma~\ref{lem:decomposition}(a).
\end{remark}

\subsection{Conditional transfer operators and the multiplicity of ACCIMs}

We complete the introduction by characterising CIMs as eigenvectors of
certain conditional transfer operators. This provides a concrete
mathematical setting for the approximation algorithms, and gives a useful 
technical tool for establishing the existence of absolutely continuous CIMs.

For each $k\geq 0$ put $m_k=m|_{A_k}$ (so that $m_0=m|_{A}$). 
Then $T:(A_{k+1},m_{k+1})\to(A_k,m_k)$ is a nonsingular transformation, so 
that $m_{k+1}\circ T^{-1} \ll m_{k}$ and a {\em conditional Frobenius--Perron
operator\/} $\L_k: L^1(A_{k+1};m_{k+1})\to L^1(A_k;m_k)$ can be defined in the
usual manner:
$$\L_kf = \frac{d}{dm_k}([f\,m_{k+1}]\circ T^{-1}).$$
Dual to $\L_k$ is the {\em (conditional) Koopman operator\/} $U_k:L^\infty(A_k;m_k)\to 
L^\infty(A_{k+1};m_{k+1})$ with the action
$$U_k\psi = \psi\circ T.$$
The relation
\begin{equation}\label{eq:dualityk}
\int_{A_k}(\L_k\,\varphi)\,\psi\,dm_k=\int_{A_{k+1}}\varphi\,U_k\psi\,dm_{k+1}
\end{equation}
is automatic
for $\varphi\in L^1(A_{k+1};m_{k+1}), \psi\in L^\infty(A_k;m_k)$. In particular,
for any $\varphi\in L^1(A;m_0)$ and $\psi\in L^\infty(A;m_0)$,
\begin{equation}\label{eqn:submarkov}
\int_{A_0} \L_0(\varphi\,\mathbf{1}_{A_1})\,\psi\,dm 
= \int_{A_1}\varphi\,U_0\psi\,dm.\end{equation}

\begin{lemma}\label{lem:lemma2} 
Let $(T,A,m|_A)$ be a nonsingular open dynamical system and let 
$\mu\ll m$ be a measure such that $\mu(A_0)=1$. Then  
is a CIM with escape rate $-\log\alpha$ if and only if 
$\L_0(\mathbf{1}_{A_1}\frac{d\mu}{dm})=\alpha\,\frac{d\mu}{dm}$.
\end{lemma}

\begin{proof}
 Let $\varphi=\frac{d\mu}{dm}$. Then for $E\subseteq A_0$,
one has $T^{-1}E\subseteq A_1$ so that, using equation~(\ref{eqn:submarkov})
$$\int_E\L_0(\mathbf{1}_{A_1}\varphi)\,dm_0
=\int_{A_1}\varphi\,U_0\mathbf{1}_E\,dm_1
=\int \varphi\,\mathbf{1}_{T^{-1}E}\,dm = \mu(T^{-1}E).$$
Since $\alpha\int_E\varphi\,dm =\alpha\,\mu(E)=\mu(T^{-1}E)$.
\end{proof}

Lemma~\ref{lem:lemma2} characterises absolutely continuous conditionally
invariant measures (ACCIMs) as those whose density functions 
solve a conditional transfer operator equation: 
$\L_0(\mathbf{1}_{A_1}\varphi)=\alpha\,\varphi$. However, in contrast to 
the typical situation for nonsingular dynamical systems, this equation
may have an uncountable number of solutions for each~$\alpha$ if no additional
regularity is specified; see~\cite[Theorem~3.1]{DY06} and discussion therein. 
We now give a version of this result.

\begin{theorem}\label{th:1}
Let $(T,A,m)$ be a nonsingular open dynamical system. 
If there is $\kappa>0$ such that 
$\L_0\mathbf{1}_{A_1}\geq \kappa\,\mathbf{1}_{H_\infty}$ and $m(H_\infty)>0$ 
then for 
every $\alpha\in(0,1)$ there is a CIM which is AC with respect to $m$ and
has escape rate $-\log\alpha$.
\end{theorem}

\begin{proof}
There is at least one $N$ for which $m(H_N\setminus K_\infty)>0$. By an 
inductive application of Lemma~\ref{lem:lemma1}(c), 
$m(H_1\setminus K_\infty)>0$.
Now let $\mu_1\ll m|_{H_1\setminus K_\infty}$ 
be a finite measure and put $\varphi_1 = \frac{d\mu_1}{dm}$. Note that
$\mathbf{1}_{A_1}\,\varphi_1=0$.
Next, we construct (inductively) a sequence
of integrable functions $\varphi_k$, supported on $H_{k}\setminus K_\infty$
such that each $\L_0(\mathbf{1}_{A_{1}}\varphi_{k+1})
=\L_k\varphi_{k+1}=\varphi_k$. Let $\varphi_k\in L^1(H_k\setminus K_\infty;m_k)$
be given. Assume that $\varphi_k$ is bounded (the general case follows from
the bounded case by an approximation argument). On $H_{k+1}\setminus K_\infty$
put
$$\varphi_{k+1}:=
\frac{\varphi_k\circ T}{U_k\L_k\mathbf{1}_{H_{k+1}\setminus K_\infty}}$$
(note that the denominator is bounded below by 
$\kappa\,\mathbf{1}_{H_{k+1}\setminus K_\infty}$).
Let $\mu_{j}=\varphi_{j}\,m_{j}$ for $j=k,k+1$ and 
$E\subseteq H_{k}\setminus K_\infty$. Then
\begin{align*}
\mu_{k+1}\circ T^{-1}E 
&= \int_{H_{k+1}\setminus K_\infty} \varphi_{k+1}\,U_k\mathbf{1}_E\,dm\\
&=\int_{A_{k+1}} U_k(\varphi_k\,\mathbf{1}_E/\L_k\mathbf{1}_{H_{k+1}\setminus K_\infty})\,\mathbf{1}_{H_{k+1}\setminus K_\infty}\,dm
=\int_{A_k}\,\varphi_k\,\mathbf{1}_E\,dm = \mu_k(E).\end{align*}
Thus, $\varphi_k=\frac{d}{dm_k}\mu_k = \frac{d}{dm_k}\mu_{k+1}\circ T^{-1}
=\L_k\frac{d\mu_{k+1}}{dm_{k+1}}=\L_k\varphi_{k+1}$. 
Using $E=H_k\setminus K_\infty$ and Lemma~\ref{lem:lemma1}(c), 
$\int\varphi_k\,dm=\int\varphi_{k+1}\,dm$.
Finally, put 
$\varphi=\frac{1-\alpha}{\mu_1(A_0)}\sum_{k=1}^\infty\alpha^{k-1}\varphi_k$. 
Then, $\int_{A_0}\varphi\,dm=1$ and
$\L_0(\mathbf{1}_{A_1}\varphi)=\alpha\,\varphi$. The theorem follows from
Lemma~\ref{lem:lemma2}.
\end{proof}

\begin{remark}
The proof given above is essentially the one from~\cite{DY06}; the different
conditions are to account for the fact that we have not imposed any
topological (or smoothness) restrictions on $T$. Note that
each choice of finite AC measure on $H_1\setminus K_\infty$ gives a 
different ACCIM.   
\end{remark}

\section{Convex optimisation for the ACCIM problem}

We now describe a selection principle for ACCIMs based on the
Shannon-Boltzmann entropy. The first idea is to encode the criteria
for being a CIM into a sequence of moment conditions, and to
search for {\em approximate\/} CIMs which locally resemble
the measure~$m$. This leads to the optimisation
problems~$(P_{n,\alpha})$, where the entropy maximising density is
sought, subject to meeting the first $n$ moment conditions for 
conditional invariance (with escape rate $-\log\alpha$).
Then, in Section~\ref{sec:convexdual}, we
recall some standard results from convex optimisation which
allow the MAXENT problem~$(P_{n,\alpha})$ to be recast in dual form. 
Theorem~\ref{th:2} identifies a condition which is both necessary and 
sufficient for solvability of the dual problem. 
Section~\ref{sec:domainreduction} introduces a {\em domain reduction\/}
technique which ensures that the conditions of Theorem~\ref{th:2} are met,
revealing an interesting connection between the structure of the moment
conditions and the backwards transient sets $K_\infty$. The main result is Theorem~\ref{th:3}: an explicit formula for the solution of ($P_{n,\alpha}$).

\subsection{Moment formulation of the ACCIM problem}

By Lemma~\ref{lem:lemma2}, if $\mu$ is an ACCIM
and $\varphi=\frac{d\mu}{dm}$ then 
$$\L_0(\mathbf{1}_{A_1}\varphi)= \alpha\,\varphi, 
\qquad \alpha=\int_{A_1}\varphi\,dm=\mu(A_1).$$
This is equivalent to
$$\int_{A_0}\left[\L_0(\mathbf{1}_{A_1}\varphi) - \alpha\,\varphi\right]\,\psi\,dm=0
\qquad \forall\psi\in L^\infty(A;m),
\qquad \int (\mathbf{1}_{A_1}\,\varphi)\,dm = \alpha$$
and hence, using equation~(\ref{eqn:submarkov}), 
$$\int_{A_0}\left[\mathbf{1}_{A_1}\psi\circ T - \alpha\,\psi\right]\varphi\,dm=0
\qquad \forall\psi\in L^\infty(A;m),
\qquad \int_{A_1} \varphi\,dm = \alpha.$$
To obtain a computationally tractable representation of these conditions,
observe that it suffices to verify for all $\psi$ in a weak* dense subset
of $L^\infty(A;m_0)$. 

\begin{definition} Let 
$\{\psi_j\}_{j=1}^\infty\subset L^\infty(A;m_0)$
be a sequence whose span is weak* dense and put $\psi_0=\mathbf{1}_{A}$. Fix $\alpha\in(0,1]$. 
Then
\begin{equation}\label{e.Fn}
\begin{split}
\mathcal{F}_n := \Bigl\{0\leq \varphi\in L^1(A;m_0)~:~ & 
\int_{A_1}\varphi\,dm=\alpha, \int_{A}\varphi\,\psi_0\,dm=1, 
\qquad\mbox{and}\\ 
&\int_{A}\left[\mathbf{1}_{A_1}\psi_j\circ T - \alpha\,\psi_j\right]
\varphi\,dm=0, j=1,\ldots,n\Bigr\}.\end{split}
\end{equation}
are {\bf approximately conditionally invariant
densities\/} with escape rate $-\log\alpha$.\end{definition}

Notice that each $\mathcal{F}_{n+1}\subset \mathcal{F}_n$. If a 
sequence $\{f_n\}$ is chosen such that each $f_n\in\mathcal{F}_n$ and
$f_n\xrightarrow{weak} f_\infty$ then $f_\infty\in\cap_{n>0}\mathcal{F}_n$.
Such an $f_\infty$ is the density of a CIM. Using arguments similar
to those leading up to Theorem~5.2 in~\cite{BM06}, one has weak (and 
indeed $L^1$) convergence of such a sequence when selecting $f_n$ to solve
\begin{equation}\tag{$P_{n,\alpha}$}
\mbox{maximize~} H(f) \quad s.t. \quad f\in\mathcal{F}_n
\end{equation}
where $H$ is a suitably chosen functional. We use the Shannon-Boltzmann
entropy
$$H(f) := - \int_A f(x)\,\log f(x)\,dm(x)$$
(where $t\,\log t$ is set to $0$ when $t=0$ and $\infty$ when $t<0$). 
If $T$ admits an ACCIM $\mu$ for which $H(\frac{d\mu}{dm})>-\infty$,
then each problem ($P_{n,\alpha}$) has a unique solution~$f_n$, and $\lim f_n$ 
exists both weakly and in $L^1$ (proofs can be adapted from~\cite{BM06}).

Each {\em primal problem\/} ($P_{n,\alpha}$) is concave, admitting a 
solution $f_{n,\alpha}$ depending on both $n$ and $\alpha$. 
As we illustrate with numerical
examples (Section~\ref{sec:eg}) the  role of $\alpha$ is interesting,
being a parameter that is tunable to produce a range of escape 
rates\footnote{The flexibility to tune~$\alpha$ without impact on 
numerical effort is reminiscent of the use of Ulam's method to calculate 
the topological pressure of piecewise smooth dynamical systems by 
varying an inverse temperature parameter~\cite{FMT07}.}: 
for $\alpha$ near $0$, escape is rapid (with mass of the ACCIM 
tending to concentrate on the first few preimages of the hole); 
for $\alpha$ near $1$, escape is slow with mass concentrated nearer 
to $\Omega$.  

In order to identify the entropy maximising ACCIM we
propose a nested approach: at the outer level, 
for each fixed~$n$, optimise $H(f_{n,\alpha})$
(over $\alpha$); as an `inner' step, each $f_{n,\alpha}$ 
is computed to solve~($P_{n,\alpha}$).

\begin{remark} The optimisation problem ($P_{n,\alpha}$) can be reformulated 
to remove~$\alpha$ as a variable. One simply replaces the $j$th moment
condition in~(\ref{e.Fn}) with
$$\int_{A_0}\left[\mathbf{1}_{A_1}\psi_j\circ T - 
({\textstyle\int_{A_1}\varphi\,dm})\,\psi_j\right]\varphi\,dm=0$$
for each $\psi_j$. This destroys the linearity of the constraint,
and potentially the convexity of the optimisation problem.
\end{remark}

\subsection{Convex duality for problem ($P_{n,\alpha}$)}
\label{sec:convexdual}

Problems like $(P_{n,\alpha})$ are never solved directly. Instead, a `Lagrange multipliers'
approach converts the problem to an equivalent finite-dimensional unconstrained 
optimisation. For the benefit of readers not familiar with this type of argument, we outline
the steps leading to this `dual formulation'.
Let $n,\alpha$ and $\{\psi_k\}_{k=1}^n$ be fixed. To simplify matters we
assume that the test functions form
a partition of unity over $A$, so $\psi_0=\mathbf{1}_A=\sum_{k=1}^n\psi_k$ 
and 
$$0=\int_{A_0}\left[\mathbf{1}_{A_1} \mathbf{1}_{A_0}\circ T- 
\alpha\,\mathbf{1}_{A_0}\right]\varphi\,dm=
\int_{A_1}\varphi\,dm - \alpha\,\int_{A_0}\varphi\,dm$$
follows from the corresponding conditions for $\psi_1,\ldots,\psi_n$. 
The normalisation $\int_{A_0}\varphi\,dm=1$ is thus a consequence of
$\int_{A_1}\varphi\,dm = \alpha$, so only one of those conditions is needed.

\begin{definition}
Define $\M:L^1(A;m_0)\to\R^{n+1}$ by 
$$(\M\varphi)_0=\int_{A_1}\varphi\,dm\qquad\mbox{and}\quad
(\M\varphi)_j = \int_{A}\left[\mathbf{1}_{A_1}\psi_j\circ T - \alpha\,\psi_j\right]
\varphi\,dm$$
for $j=1,\ldots,n$. Let $\M^*:\R^{n+1}\to L^\infty(A;m_0)$ be defined by
$$\M^*\lambda = \lambda_0\,\mathbf{1}_{A_1} 
+ \sum_{j=1}^n\lambda_j(\mathbf{1}_{A_1}\psi_j\circ T-\alpha\,\psi_j).$$
Let $\bfe=[1,0,\ldots,0]^T\in\R^{n+1}$, put 
$Q(\lambda):= \alpha\,\lambda^T\,\bfe - \int_{A}\exp(\M^*\lambda-1)\,dm$ 
and define a {\em dual problem\/}:
\begin{equation}\tag{$D_{n,\alpha}$}
\mbox{maximise}~Q(\lambda)\qquad \mbox{s.t.~}\lambda\in\R^{n+1}.
\end{equation}
\end{definition}

We now outline how $(D_{n,\alpha})$ is related to ($P_{n,\alpha}$). 
First, note that 
$$f\in\mathcal{F}_n \Leftrightarrow \M f=
\alpha\,\bfe\qquad\mbox{and}\qquad \lambda^T(\M f) =
\int_{A}\M^*\lambda\,f\,dm\qquad\forall f\in L^1(A;m).$$
For every $\lambda\in\R^{n+1}$
\begin{align*}
\sup_{f\in\mathcal{F}_n} H(f) &= \sup_{\{f~:~\M f=\alpha\,\bfe\}} H(f)\\
&= \sup_{\{f~:~\M f=\alpha\,\bfe\}} [H(f) + \lambda^T(\M f - \alpha\,\bfe)]\\
&\leq \sup_{f\in L^1(A;m)} [H(f) + \lambda^T(\M f - \alpha\,\bfe)]\\
&= -\alpha\,\lambda^T\,\bfe + \sup_{f\in L^1(A_0;m)} \left[\int_{A}\M^*\lambda\,f\,dm - (-H(f))\right]\\
&=  -\alpha\,\lambda^T\,\bfe + H^*(\M^*\lambda)\\
&= -\alpha\,\lambda^T\,\bfe + \int_{A}\exp(\M^*\lambda - 1)\,dm = -Q(\lambda)\\
\end{align*}
where $H^*$ is the {\em Fenchel conjugate\/} of the convex functional~$-H$, and the second to last
equality is a nontrivial result in convex analysis (see Rockafellar~\cite{R1} and 
Borwein and Lewis~\cite{BL91b}). Observe that $-Q(\lambda)$ is 
an upper bound on $H(f)$ for all $f\in\mathcal{F}_n$ and $\lambda\in\R^{n+1}$ so that
the (negative of) the solution to ($D_{n,\alpha}$) provides an upper bound on the solution to ($P_{n,\alpha}$).
This is called the {\em principle of weak duality\/}. In fact, ($D_{n,\alpha}$) is a differentiable,
unconstrained, concave maximisation problem, and our method involves solving it.


\begin{theorem}[Dual attainment]\label{th:2}  Let $\alpha,n$ be fixed.
\begin{enumerate}
\item $\lambda^*$ solves ($D_{n,\alpha}$) if and only if  
$f_n:=\exp(\M^*\lambda^*-1)\in\mathcal{F}_n$ and $H(f_n)=-Q(\lambda^*)$;
\item the problem ($D_{n,\alpha}$) attains its maximum if and only if
\begin{equation}\label{eq:condition}
0\neq\lambda\in\{\ker{\M^*}\oplus\mbox{span}(\bfe)\}^\perp 
\quad\Rightarrow \quad [\M^*\lambda]^+\neq 0~m\mbox{-a.e.}.
\end{equation}
\end{enumerate}
\end{theorem}

\begin{proof}
(a) This is a standard result in dual optimisation theory, and is a consequence of the 
fact that $\lambda^*$ solves ($D_{n,\alpha}$) iff
$\alpha\,[\bfe]_j-[\M\exp(\M^*\lambda^*-1)]_j 
= \frac{\partial Q}{\partial \lambda_j}|_{\lambda^*}=0$ for $j=0,\ldots,n$. 
\newline
(b) Sufficiency
of~(\ref{eq:condition}) is established by minor modifications to the proof of Theorem~3.3 in~\cite{BM07}. 
For necessity, suppose that $\lambda^T\bfe=0$, 
$0\neq \lambda\in \{\ker{\M^*}\}^\perp$ and $M^*\lambda\leq 0$. 
Then there are $\kappa>0$ and
$E\subseteq A$ such that $m(E)>0$ and 
$M^*\lambda \leq-\kappa\,\mathbf{1}_E$. Then, 
for any $\lambda^\dagger\in\R^{n+1}$ and $t>0$,
$$Q(\lambda^\dagger+t\,\lambda)\geq Q(\lambda^\dagger) 
+ (1-e^{-\kappa t})\int_E\exp(\M^*\lambda^\dagger -1)\,dm > Q(\lambda^\dagger).$$
Hence $Q$ cannot attain its maximum.
\end{proof}

\subsection{Domain reduction and dual optimality conditions}
\label{sec:domainreduction}

The condition~(\ref{eq:condition}) incorporates some important facts about 
ACCIMs. First, by Theorem~\ref{th:1}, there exist ACCIM. It follows from this 
that $\mathcal{F}_n\neq\emptyset$ and 
$\alpha\,\bfe\in\Range{\M}=\{\ker{\M^*}\}^\perp$ 
(this is the reason for separating
out the direction~$\bfe$). Second, the support of each ACCIM must be disjoint
from subsets of~$A$ associated with ``bad functions''. (This is made precise
in Lemma~\ref{lem:ignorebad} below.) A function $\psi$ will be called
a {\em bad function\/} if $\mathbf{1}_{A_1}\,\psi\circ T-\alpha\,\psi\leq 0$
(but not equal to $0$ $m$--a.e.). If $\lambda\in\R^{n+1}$ is such that 
$[\lambda]_0=0$ and $\M^*\lambda\leq 0$ (but nonzero), then 
$\psi=\sum_{j=1}^n[\lambda]_j\,\psi_j$ is a {\em bad function\/}. 
The condition~(\ref{eq:condition}) for solvability of ($D_{n,\alpha}$) 
is equivalent to there being no
bad functions in $\mbox{span}\{\psi_j\}_{j=1}^n$. We are going to show that
bad functions may exist (Example~\ref{eg:bad}), but they are 
irrelevant to the ACCIMs (their supports are disjoint from $H_\infty$; see 
Lemma~\ref{lem:decomposition}(c) and Lemma~\ref{lem:ignorebad}) 
and can be excised from the problems ($P_{n,\alpha}$) and ($D_{n,\alpha}$) 
(Lemma~\ref{lem:reduce}). We call this latter procedure {\em domain reduction\/}.

\begin{example}\label{eg:bad}
If $x\in \cup_{n\geq 0}K_n$ let $N(x):=\min\{k~:~T^{-k}(x)\cap A_0=\emptyset\}$.
Note that $N(x)+1\leq N(T(x))$ (where $N(y)=\infty$ if 
$y\notin \cup_{n\geq 0} K_n$). Define $\psi(x)=(\alpha/2)^{N(x)}$. Then
$-(\alpha/2)\psi=(\alpha/2)\psi - \alpha\psi \geq \psi\circ T-\alpha\psi$. 
Hence 
$\mathbf{1}_{A_1}\psi\circ T - \alpha \psi<0$ on $\cup_{n\geq 0}K_n$.  
\end{example}

\begin{lemma}\label{lem:ignorebad}
Let $\alpha\in(0,1)$ and 
suppose that $\psi\in L^\infty(A;m)$ 
satisfies $\mathbf{1}_{A_1}\psi\circ T\leq \alpha\,\psi$.
Then $\psi|_{\cup_{k>0}H_k}\geq 0$ and $\psi|_{A\setminus K_\infty}\leq 0$.
In particular, $m(H_\infty\cap \mbox{supp}(\psi))=0$.
\end{lemma}

\begin{proof} First, let $x\in H_1$. Then $\mathbf{1}_{A_1}(x)=0$ so 
$0=\mathbf{1}_{A_1}\psi\circ T(x)\leq \alpha\,\psi(x)$, so $\psi|_{H_1}\geq 0$. Now suppose that 
$x\in H_k$. Then $T^{k-1}(x)\in H_{1}$ so that 
$$0\leq \psi(T^{k-1}(x))\leq \alpha\,\psi(T^{k-2}(x)) \leq 
\cdots \leq \alpha^{k-1}\psi(x).$$
Thus, $\psi|_{H_k}\geq 0$. On the other hand, if $x\notin K_\infty$ 
then for each $k>0$ 
there is at least one $x_{-k}$ such that $T^k(x_{-k})=x$. 
Then $\psi(x)=\psi\circ T^{k}(x_{-k})\leq \alpha^k\,\psi(x_{-k})
\leq \alpha^k\|\psi\|_\infty$. Letting $k\rightarrow\infty$, $\psi(x)\leq 0$.
\end{proof}

To apply Theorem~\ref{th:2} when $K_\infty\neq\emptyset$ 
we need to ensure that the chosen test functions
$\{\psi_j\}_{j=1}^n$ are {\em unable to detect\/} bad functions. To do this, we exploit a {\bf basis
specific domain reduction\/}: remove from the domain~$A$ the support of any function
$h=\M^*\lambda$ where $h\leq 0$ and $\lambda\in\Range{\M}/\mbox{span}\{\bfe\}$. Let $\hat{A}$ denote this
reduced domain.

\begin{lemma}\label{lem:reduce}
In the notation of this section, suppose that $\hat{A}$ is measurable and $f\in\mathcal{F}_n$. Then
$f=f\,\mathbf{1}_{\hat{A}}$ $m$--a.e.
\end{lemma}

\begin{proof} Suppose that $m(\mbox{supp}(f)\setminus \hat{A})>0$ and let
$\lambda$ be such that $\lambda^T\bfe=0$, $\M^*\lambda\leq 0$ and 
$\mbox{supp}(\M^*\lambda)\cap\mbox{supp}(f)\subseteq A_0\setminus \hat{A}$ has positive measure. 
Then, $\M f=\alpha\,\bfe$ so that 
$0=\lambda^T(\M f) = \int_{A_0}\M^*\lambda\,f\,dm <0$, an obvious contradiction.
\end{proof}

In view of Lemma~\ref{lem:reduce}, $m$ can be replaced with $\hat{m}=m|_{\hat{A}}$ in the definition of
the problem ($P_{n,\alpha}$) {\em without any change to the set $\mathcal{F}_n$\/}. The value of the
problem is also unchanged, since there is no contribution to $H(f)$ from those places where
$f$ takes the value $0$. The duality theory is now applied to the measure space $(A_0,\hat{m})$,
and the corresponding dual problem is
\begin{equation}\tag{$\hat{D}_{n,\alpha}$}
\mbox{maximise~}\hat{Q}(\lambda):=
\alpha\,\lambda^T\,\bfe - \int_{\hat{A}}\exp(\M^*\lambda-1)\,dm\qquad
\mbox{s.t.~}\lambda\in\R^{n+1}.
\end{equation}
Notice that if $\M^*\lambda\leq 0$ $m$--almost everywhere, then the domain 
reduction ensures that $\M^*\lambda=0$ $\hat{m}$--a.e. Thus, all potentially problematic
$\lambda$ have been pushed into $\ker{\M^*}$ (modulo $\hat{m}$). In particular, condition~(\ref{eq:condition})
is satisfied for the reduced domain. The previous results can be collected in our main
theorem. 

\begin{theorem}\label{th:3}
Let $\alpha,n$ be fixed and suppose that $\hat{A}$ is measurable. Then ($\hat{D}_{n,\alpha}$) attains its maximum
at finite $\lambda^*$ and $f_n=\mathbf{1}_{\hat{A}}\exp(\M^*\lambda^*-1)$ solves ($P_{n,\alpha}$).
\end{theorem}

We note that $\M^*$ may have nontrivial kernel (modulo $\hat{m}$), so the optimising $\lambda^*$ 
can be non-unique. We also make the following observations:
\begin{itemize} 
\item the reduced domain~$\hat{A}$ depends on~$n$, possibly $\alpha$ and may be {\em very difficult to determine for general test functions\/};
\item assuming the escape~hypothesis~(\ref{eqn:escape}) we have 
$A\setminus \hat{A}\subseteq K_\infty \pmod{m}$ [$m(A_\infty)=0$ by~($\ref{eqn:escape}$) which together with Lemma~\ref{lem:ignorebad} shows that $\mbox{supp}(\psi)\subseteq K_\infty\pmod{m}$ for any bad function~$\psi$; the observation follows];
\item if $\hat{A}$ is overestimated then condition~(\ref{eq:condition}) fails and
the dual optimisation problem does not have a solution for finite~$\lambda$.   Nevertheless it would be a simple matter to set up the dual formulation
$(D_{n,\alpha})$ and seek a numerical `solution' of this infeasible optimization problem without first verifying the optimality condition in equation (\ref{eq:condition}); such a numerical approach is bound to be both unstable and misleading. 
See Borwein and Lewis~\cite{BL91b} for further discussion of this and related issues.
\end{itemize}
Nothwithstanding these warnings, in Section~\ref{sec:three} we show how to  compute~$\hat{A}$ for
piecewise constant test functions based on a measurable partition of~$A$.

\section{A MAXENT procedure for approximating ACCIMs}
\label{sec:three}

Under the conditions of Theorem~\ref{th:1} there are many 
ACCIMs for each escape rate. If at least one of these has a
density with finite Shannon-Boltzmann entropy then the solutions
of a sequence of problems ($P_{n,\alpha}$) will converge (in $L^1$) 
as $n\to\infty$ to the density of an ACCIM. This, in principle,
allows one to select an ``entropy maximising'' ACCIM; the entropy
maximisation spreads mass as uniformly as possible, given the 
 condition of being a CIM. Solutions to each problem ($P_{n,\alpha}$) 
can be calculated via convex duality, provided there are no
``bad functions'' ($\M^*\lambda$ which fail the 
condition~(\ref{eq:condition}) in Theorem~\ref{th:2}). This condition
can be ensured by a basis dependent domain reduction (Lemma~\ref{lem:reduce}
and Theorem~\ref{th:3}), leading to a domain reduced dual 
problem~($\hat{D}_{n,\alpha}$). We now make a specific choice of test functions, 
reminiscent of Ulam's method~\cite{ulam60,FD03,F2001}. We 
identify the reduced domain~$\hat{A}$ (Lemma~\ref{lem:reduceddomain}), derive the relevant optimality equations (Lemma~\ref{lem:opteq}) and present a convergent fixed point method for their solution.

\subsection{Piecewise constant test functions and domain 
reduction}

Let $\{\psi_j\}$ be obtained from a sequence of increasingly 
fine partitions of $A$. 
In particular, let $\mathcal{B}_n$ be a partition of $A$ into
measurable subsets $\{B_1,\ldots, B_n\}$ and put $\psi_j=\mathbf{1}_{B_j}$. 
Notice that $\mathbf{1}_{A}=\sum_{j=1}^n\psi_j$ so the 
partition of unity assumption is satisfied ({\em c.f.} Section~\ref{sec:convexdual}).
To derive and solve the optimality equations for ($\hat{D}_{n,\alpha}$), notice that $\M^*\lambda$ is a piecewise
constant function, on elements of  $\mathcal{B}_n\vee \{T^{-1}\mathcal{B}_n,H_1\}$:   
\begin{align}\label{eq:mstarl}
\M^*\lambda &= 
\mathbf{1}_{A_1}\sum_{j,k=1}^n(\lambda_0+\lambda_j-\alpha\,\lambda_k)\mathbf{1}_{B_j}\circ T\mathbf{1}_{B_k}
+\mathbf{1}_{H_1}\sum_{k=1}^n(-\alpha\,\lambda_k)\,\mathbf{1}_{B_k}\nonumber \\
&=\sum_{j,k=1}^n(\lambda_0+\lambda_j-\alpha\,\lambda_k)\mathbf{1}_{B_k\cap T^{-1}B_j}
-\alpha\,\sum_{k=1}^n\lambda_k\,\mathbf{1}_{H_1\cap B_k}
\end{align}
(note that $\mathbf{1}_{A_1}=\mathbf{1}_{A\cap T^{-1}A}
=\sum_{jk}\mathbf{1}_{B_k\cap T^{-1}B_j}$).

\begin{definition}\label{def:MX} 
For the partition~$\mathcal{B}_n$, form a matrix $C$ and vector $\mathbf{c}$ by putting
$$C_{kj}=m(B_k\cap T^{-1}B_j)\qquad\mbox{and}\qquad c_k=m(H_1\cap B_k)\qquad j,k=1,\ldots,n.$$
A set $B_j$ is {\bf reachable\/} from $B_k$ if there 
is $n>0$ such that $(C^n)_{kj}>0$; write $k\leadsto j$.
\end{definition}

\begin{remark} The entries of the matrix $C$ are the same data needed to compute the (sub)stochastic transition matrices used by Ulam's method.\end{remark}

\begin{lemma}\label{lem:reduceddomain}
Suppose that $(T,A,m)$ is a nonsingular open 
dynamical system and that~$m(A_\infty)=0$.
Fix $\alpha,n$ and let $\hat{A}$ be the reduced domain 
when $\M^*$ is constructed from the partition~$\mathcal{B}_n$.
Then $\hat{A}$ is the union of those $B_k$ where either $k\leadsto k$ 
or there is at least one $i$ for which $i\leadsto i\leadsto k$; 
in particular, $\hat{A}$ is measurable. 
\end{lemma}

\begin{proof} Let $\lambda^T\,\bfe=0$ and suppose that $\M^*\lambda\leq 0$. 
From equation~(\ref{eq:mstarl}), we immediately have
$$\lambda_j\leq\alpha\lambda_k \quad\mbox{when~} C_{kj}>0\qquad\mbox{and}\qquad
\lambda_k\geq0\quad\mbox{when~}c_k>0.$$
Since $C$ is a non-negative matrix, $i\leadsto k$ iff there is a string $i=i_0,i_1\ldots,i_n=k$
such that each $C_{i_li_{l+1}}>0$. Thus, by induction, if $i\leadsto k$ then there is an $n>0$ 
such that $\lambda_k\leq \alpha^n\,\lambda_i$. 
First, if $c_k>0$ and $i\leadsto k$ we infer that
$\lambda_i\geq 0$. Next, since $m(A_\infty)=0$, for every $B_i$ there is an $n$ for which 
$m(B_i\cap H_n)>0$. Then, since $T$ is nonsingular, there is $B_l$ such that $C_{il}>0$ and
$m(B_l\cap H_{n-1})>0$. By induction, there is a $k$ for which $i\leadsto k$ and $c_k>0$.
Hence, $\lambda_i\geq 0$ for all $i$. Now, if $k\leadsto k$, again use the inequality
$\lambda_k\leq \alpha^n\lambda_k$ to infer that   $\lambda_k\leq 0$ and hence $\lambda_k=0$.
Similarly, if $i\leadsto i \leadsto k$, $\lambda_k\leq \alpha^n\,\lambda_i=0$, so also $\lambda_k=0$.
Suppose that $k$ is one of the indices identified in the statement of the lemma. Then~(\ref{eq:mstarl})
implies that $\mathbf{1}_{B_k}\M^*\lambda = \sum_j\lambda_j\mathbf{1}_{B_k\cap T^{-1}B_j}\geq 0$;
since $M^*\lambda\leq 0$, $B_k\cap \mbox{supp}(\M^*\lambda)=\emptyset$. To complete the
proof, let $\mathcal{K}$ denote those $\hat{k}$ which fail the condition in the statement. For each such
$\hat{k}$, let $N(\hat{k})=\max\{N~:~(C^N)_{i\hat{k}}>0 ~\exists i\}$; $N(\hat{k})$ may be $0$. 
(Note that if $(C^{N})_{i{k}}>0$ for $N>n$ then there is a sequence $i=i_0,i_1,\ldots,i_n={k}$ for which $C_{i_li_{l+1}}>0$; this list must contain at least one repeat, implying ${k}\notin\mathcal{K}$.)
Note that if $C_{i\hat{k}}>0$ then $N(i)+1\leq N(\hat{k})$. Finally, for each $\hat{k}\in\mathcal{K}$ put $\lambda_{\hat{k}}=(\alpha/2)^{N(\hat{k})}$, with $\lambda_k=0$ for $k\notin\mathcal{K}$. Then,
$C_{i\hat{k}}>0$ implies $\lambda_i(\alpha/2)\geq \lambda_{\hat{k}}$.
Hence $\lambda_{\hat{k}}-\alpha\,\lambda_i\leq - \lambda_{\hat{k}}<0$. It follows that $\mbox{supp}(\M^*\lambda)=\cup_{\hat{k}\in\mathcal{K}}B_{\hat{k}}$. 
\end{proof}

\begin{remark} The set $\hat{A}$ identified by the lemma is the union of 
all $B_k$ which are reachable from the strongly connected components of the directed graph
implied by the non-zero elements of the matrix $C$. This can be found quickly and easily. 
\end{remark}

Now, form the matrix $\hat{C}$ and vector $\hat{\mathbf{c}}$ by retaining 
those entries  where $B_k$ is identified as belonging to $\hat{A}$, and 
setting the rest to $0$. These ingredients can be used to 
obtain explicit formulae for the optimality conditions for ($\hat{D}_{n,\alpha}$). 
Using equation~(\ref{eq:mstarl}), 
$$\hat{Q}(\lambda)=\alpha\lambda_0 - \sum_{jk}\exp(\lambda_0-1+\lambda_j-\alpha\,\lambda_k)\hat{C}_{kj} -\sum_k\exp(-1-\alpha\,\lambda_k)\,\hat{c}_k.$$
Because $\hat{Q}$ is differentiable and concave, the maximising $\lambda^*$ 
is found by solving the first order conditions 
$\frac{\partial\hat{Q}}{\partial\lambda_i}=0$. The following lemma
writes these conditions in a more convenient form.


\begin{lemma}\label{lem:opteq}
Assume the conditions of Lemma~\ref{lem:reduceddomain} and let $\hat{A}$ be
as given there. Let $\hat{C},\hat{\mathbf{c}}$ be obtained similarly
to Definition~\ref{def:MX}, but using $\hat{m}=m|_{\hat{A}}$ in place of $m$.
If $\{x_i\}_{i=1}^n$ are positive numbers solving
$$x_i^{1+\alpha} = \alpha\,\frac{\sum_j \hat{C}_{ij}x_j +\hat{c}_i}{
\sum_k \hat{C}_{ki} x_k^{-\alpha}}
$$
and $\lambda_0^*$ satisfies $e^{\alpha\,\lambda_0^*-1}\sum_j\hat{C}_{ij}x_jx_i^{-\alpha}=\alpha$ then $\lambda_i^*:=\log(x_i)-\lambda_0^*$ give the solution
to ($\hat{D}_{n,\alpha}$).
\end{lemma}

\begin{proof} By differentiation, the optimality equations for 
($\hat{D}_{n,\alpha}$) are
$$
\begin{array}{rcrlr}
 0 &=&\alpha-\sum_{jk}\exp(\lambda_0-1+\lambda_j-\alpha\,\lambda_k)\hat{C}_{kj}
&&(i=0)\phantom{.}\\[0.5em]
0 &=& \alpha \sum_{j}\exp(\lambda_0-1+\lambda_j-\alpha\,\lambda_i)\hat{C}_{ij}
&- \sum_{k}\exp(\lambda_0-1+\lambda_i-\alpha\,\lambda_k)\hat{C}_{ki}\\
&&&+\alpha\,\exp(-1-\alpha\,\lambda_i)\,\hat{c}_i &(1\leq i\leq n).
\end{array}
$$
The $i=0$ equation is a normalisation. By putting $x_i=e^{\lambda_i+\lambda_0}$
for $1\leq i\leq n$ the latter equations are equivalent to
$$0 = \alpha\,\sum_j \hat{C}_{ij}x_jx_i^{-\alpha} 
- \sum_k \hat{C}_{ki} x_ix_k^{-\alpha} +\alpha\,\hat{c}_ix_i^{-\alpha}.$$
Multiplying by $x_i^{\alpha}$ and rearranging gives the equations in the
statement of the lemma.
\end{proof}


\subsection{Iterative solution of the optimality equations}

We now summarise the numerical method.

\renewcommand\theenumi{\arabic{enumi}}

\begin{enumerate}
\item Specify $\alpha$ ($=e^{-\rho}$ where $\rho$ is the preferred escape rate).
\item Fix a measurable partition $\mathcal{B}_n=\{B_j\}_{j=1}^n$ of $A$.
\item Obtain the matrix $C$ and vector $\mathbf{c}$ of partition overlap masses
(as specified in Definition~\ref{def:MX}).
\item Use Lemma~\ref{lem:reduceddomain} to identify $\hat{A}$ and thus
form the dual problem~($\hat{D}_{n,\alpha}$).
\item Solve the optimality equations via Lemma~\ref{lem:opteq}. This can 
be accomplished with a fixed point iteration: set $\mathbf{x}_0=[1,\ldots,1]^T$
and iterate
$$\mathbf{x}_{t+1} = \Psi(\mathbf{x}_t)\qquad\mbox{where}\qquad
[\Psi(\mathbf{x})]_i =  \left( \alpha\,\frac{\sum_j \hat{C}_{ij}x_j +
\hat{c}_i}{\sum_k \hat{C}_{ki} x_k^{-\alpha}}\right)^{1/(1+\alpha)}$$
until desired accuracy is achieved.
\item Recover the optimal $\lambda^*$ via Lemma~\ref{lem:opteq} and
solution $f_{n,\alpha}$ to ($P_{n,\alpha}$) from Theorem~\ref{th:3}.
\item (Optional) Calculate $H(f_{n,\alpha})$.
\end{enumerate}

\subsection*{Sketch proof of convergence of the fixed point iteration}

Assume the escape hypothesis~(\ref{eqn:escape}).

Without loss of generality, assume that all sums in the definition of $\Psi$ 
are nonempty\footnote{Note that $\hat{C}_{ki}=0$ $\forall k$ only if 
$B_i\cap \hat{A}=\emptyset$. In this case also each $\hat{C}_{ij}=\hat{c}_i=0$
and the value of $\M^*\lambda$ on $B_i$ is irrelevant to the solution of
($P_{n,\alpha}$) (by Lemma~\ref{lem:reduce}). The function $\Psi$ can be defined
to be $1$ on such coordinates.}. Because ($\hat{D}_{n,\alpha}$) 
actually has a solution,
there is $\mathbf{y}^*$ for which $\Psi(\mathbf{y}^*)=\mathbf{y}^*$. 
For any $\mathbf{x}\in\R_+^n$ let 
$$V(\mathbf{x}) = \min\left\{R~:~ 
\frac{1}{R}\leq \frac{x_i}{y_i^*}\leq R, 1\leq i\leq n\right\}.$$
Clearly $V(\mathbf{x})\geq 1$ and $V(\mathbf{x})=1$ 
iff $\mathbf{x}=\mathbf{y}^*$. Moreover,
\begin{equation}\label{e.psi}
[\Psi(\mathbf{x})]_i \leq \left(\alpha\,\frac{V(\mathbf{x})\sum_j 
\hat{C}_{ij}y_j^* + \hat{c}_i}{V(\mathbf{x})^{-\alpha}\sum_k \hat{C}_{ki} 
(y^*_k)^{-\alpha}}\right)^{1/(1+\alpha)}\leq 
V(\mathbf{x})\,[\Psi(\mathbf{y}^*)]_i=V(\mathbf{x})\,{y}^*_i.
\end{equation}
Together with a similar inequality involving $1/V$, one has
$V\circ \Psi\leq V$. Thus $\{V\circ\Psi^t(\mathbf{x}_0)\}$ 
is a decreasing sequence, bounded below by $1$. 
Because $V(\mathbf{x}_0)<\infty$, all $\{\mathbf{x}_t\}$ are
confined to a closed, bounded rectangle in $\mathbb{R}^n$; let
$\mathbf{x}_*$ be a limit point of $\{\mathbf{x}_t\}$. Then 
$V\circ\Psi(\mathbf{x}_*)=V(\mathbf{x}_*)$.

Suppose that $i$ is such that\footnote{A similar argument works if
$i$ is such that  $[\Psi(\mathbf{x}_*)]_{i}=y^*_{i}/V(\mathbf{x}_*)$.}  
$[\Psi(\mathbf{x}_*)]_{i}=V(\mathbf{x}_*)y^*_{i}$. 
An inductive argument (using the equality form of (\ref{e.psi})) 
shows that $[\mathbf{x_*}]_k = V(\mathbf{x}_*)y^*_k$ and 
$\hat{c}_k=V(\mathbf{x}_*)\hat{c}_k$ whenever $i\leadsto k$. 
Since there is at least one $k$ with $\hat{c}_k>0$ reachable from $i$, 
$V(\mathbf{x}_*)=1$. Thus $\mathbf{x}_*=\mathbf{y}^*$ 
and $\mathbf{x}_t\to\mathbf{y}^*$.

\subsection{Examples}\label{sec:eg}

We present two simple examples to demonstrate the effectiveness of the
method; each implementation takes only a few dozen lines of {\sc Matlab\/}
code.

\begin{example}[Tent-map with slope 3]\label{eg:tent}
Let $X=\R$, $A=[0,1]$ and put
$$T(x) = \left\{\begin{array}{ll} 3\,x & x < 0.5 \\ 3\,(1-x) &x > 0.5 \end{array}\right.$$
Then, $A_1= [0,1/3]\cup [2/3,1]$ and $H_1=(1/3,2/3)$.
The ``natural'' ACCIM is Lebesgue measure with density $f_*=1$, 
and corresponding value of $\alpha=2/3$. 
In this case, $K_n=\emptyset=K_\infty$ (for all $n$)
and the survivor set $\Omega=A_\infty$ is the usual middle thirds Cantor set.
At a selection of values of  $\alpha\in(0,1)$ we applied the MAXENT method 
using the partition based test 
functions~$\{\psi_j= \mathbf{1}_{[(j-1)/1000,j/1000)}\}_{j=1}^{1000}$. 
The results are 
depicted in Figure~\ref{fig:tent}. As expected, for small values of $\alpha$,
escape is rapid and the ACCIMs are strongly concentrated on the hole $H_1$ and
its first few preimages. For $\alpha$ near 1, escape is slow and the ACCIMs
are more strongly concentrated around the repelling Cantor set~$A_\infty$;
see Figure~\ref{fig:tentb}.
The MAXENT method can be tuned to produce a ``most uniform'' 
approximate ACCIM, and the maximal entropy solution is in fact the 
constant density function, appearing at $\alpha=2/3$.

\begin{figure}
\begin{center}
\includegraphics[width=0.8\textwidth]{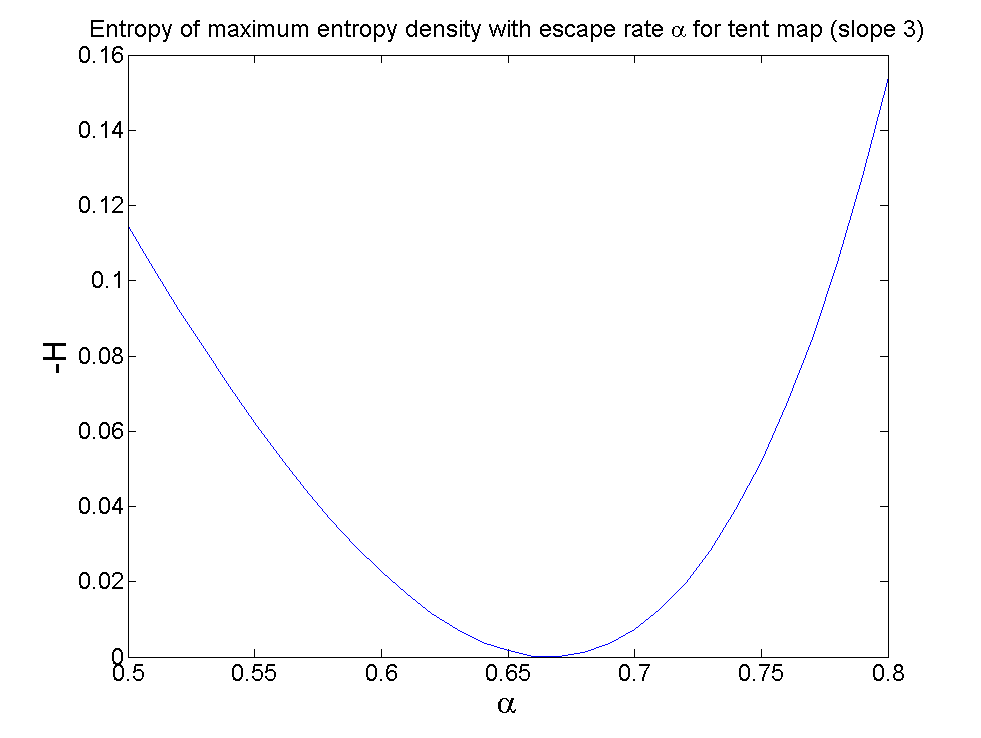}

\includegraphics[width=\textwidth]{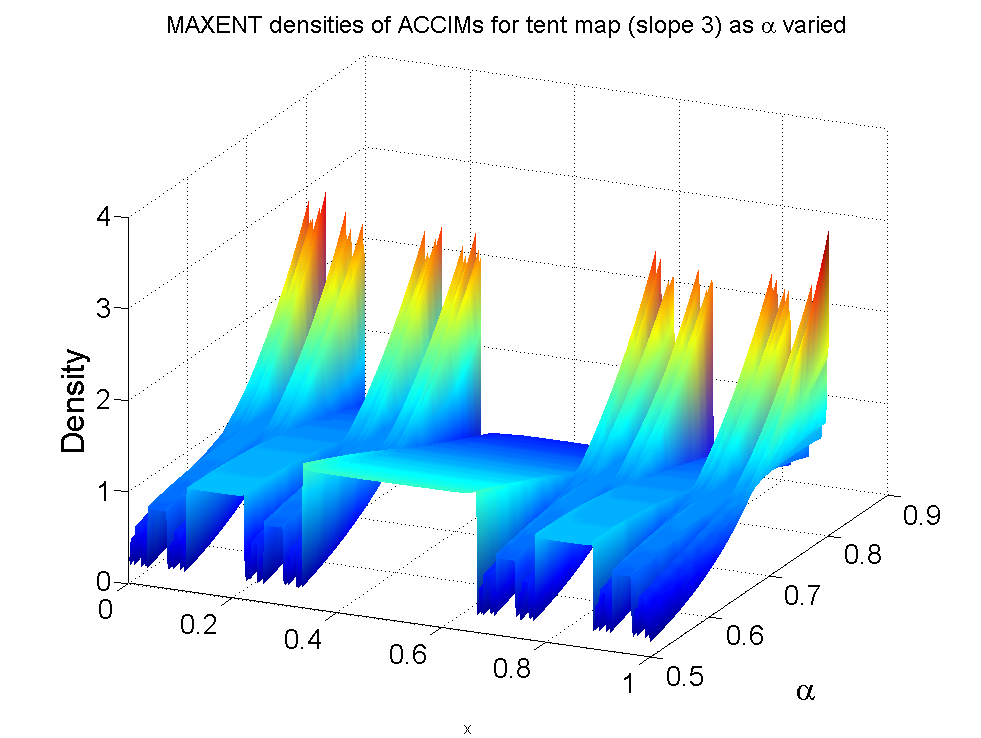}
\end{center}
\caption{Example~\ref{eg:tent}.
Above: (neg)entropy $-H(f_{n,\alpha})$ of slope~$3$ tent map ACCIMs, 
depending on $\alpha$ computed via MAXENT with uniform $n=1000$ subinterval
partition of $[0,1]$. Below: densities of the computed ACCIMs as a function
of $x\in[0,1]$ and $\alpha$.}
\label{fig:tent}
\end{figure}

\begin{figure}
\begin{center}
\includegraphics[width=0.9\textwidth]{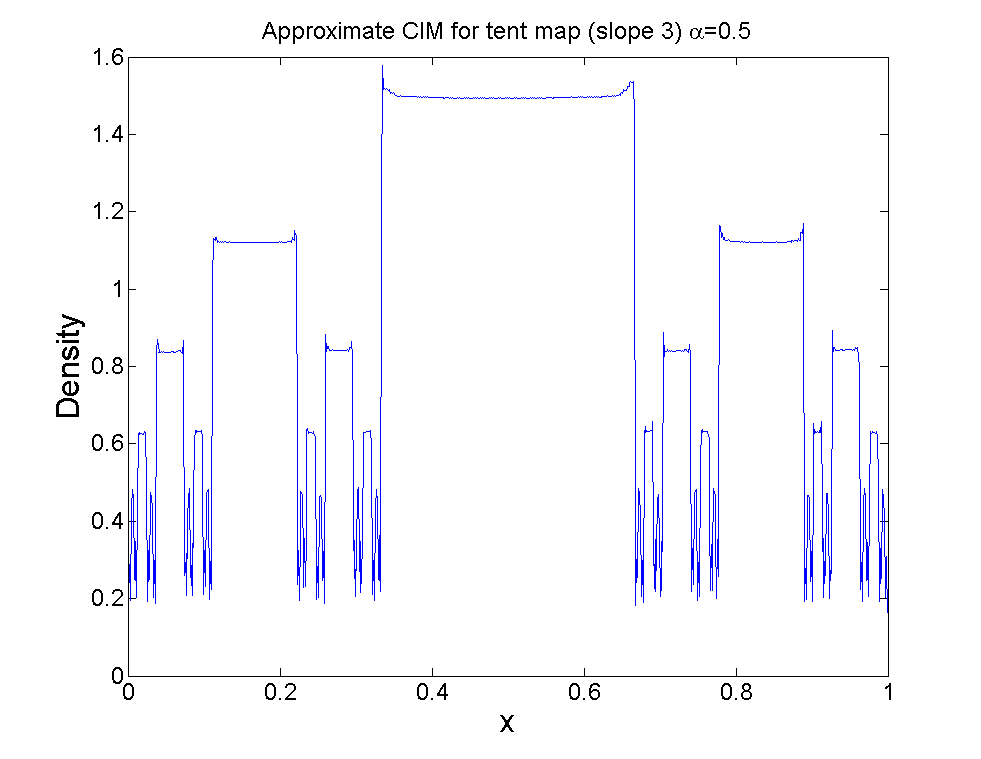}

\includegraphics[width=0.9\textwidth]{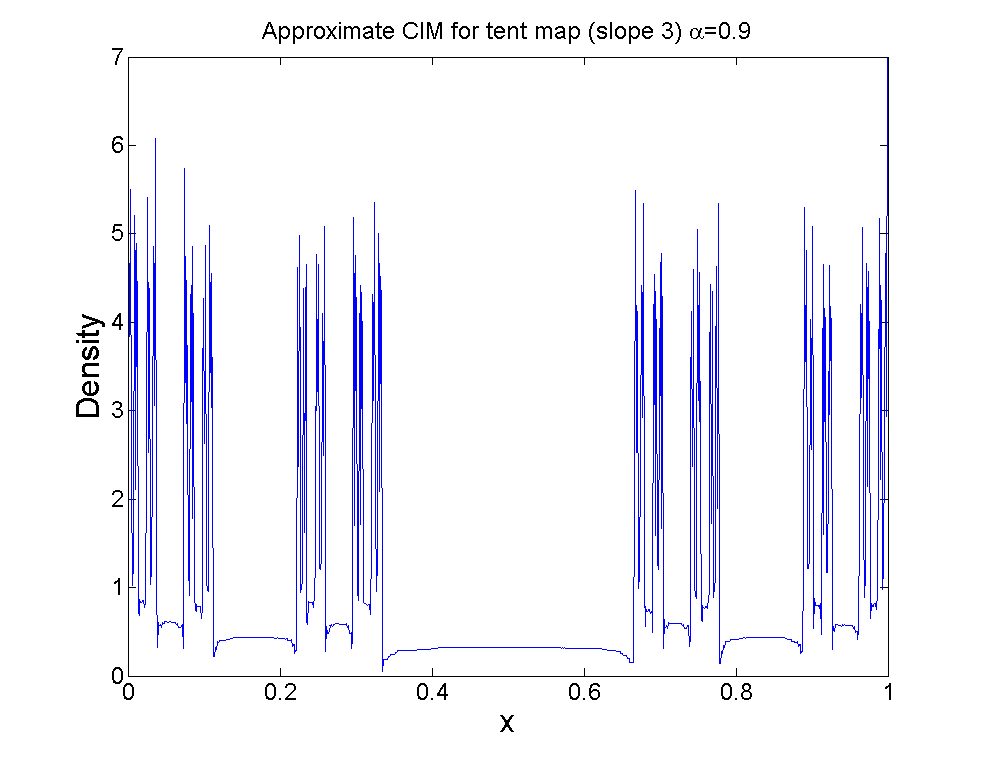}
\end{center}
\caption{Example~\ref{eg:tent} (compare Figure~\ref{fig:tent}).
Above: approximate density $f_{1000,0.5}$ of slope~$3$ tent map; note the
concentration of mass on $H_1=[1/3,2/3]$ and its preimages. Below:
approximate density $f_{1000,0.9}$ of slope~$3$ tent map; note the
concentration of mass on the survivor Cantor set $A_\infty$.}
\label{fig:tentb}
\end{figure}
\end{example}

\begin{example}[A linear saddle]\label{eg:saddle}
Let $A=[-1,1]^2$ and $m$ Lebesgue measure on $X=\R^2$; 
put $T(x,y)=(2\,x,0.8\,y)$.
Then $K_n=[-1,1]\times \pm(0.8^{(n+1)},0.8^{n}]$, $A_\infty=\{0\}\times [-1,1]$
and $H_\infty=[-1,1]\times\{0\} \setminus (0,0)$.
This linear map has a saddle-type fixed point at~$(0,0)$. The only invariant
measure is the delta~measure at $0$. All conditionally invariant measures
are supported on the local unstable manifold to the origin; in this case,
the segment of the $x$--axis contained in~$A$.
Indeed,  $m(H_\infty)=0$ and there are no ACCIMs. There are,
however, many CIMs which are AC with respect to the 
one-dimensional Lebesgue measure on the $x$-axis, and these are
detected by the numerical method. The domain reduction to $\hat{A}$
is nontrivial here, leading to a localisation in support of the MAXENT 
approximations. Calculations were performed for several~$\alpha$,
with $10000$ test functions being the characteristic functions of a 
$100\times100$ subdivision of~$A$; in this case the set 
$\hat{A}=[-1,1]\times[-0.08,0.08]$. Some CIM estimates are presented in 
Figures~\ref{fig:saddlea} and~\ref{fig:saddleb}.

\begin{figure}
\begin{center}
\includegraphics[width=0.9\textwidth]{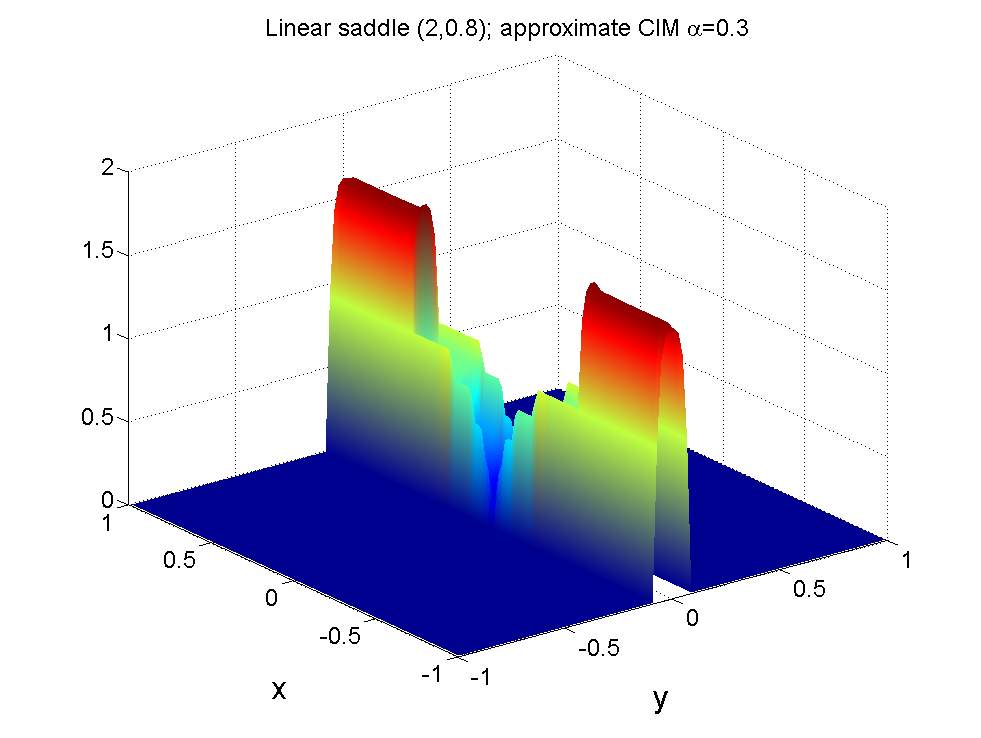}

\includegraphics[width=0.9\textwidth]{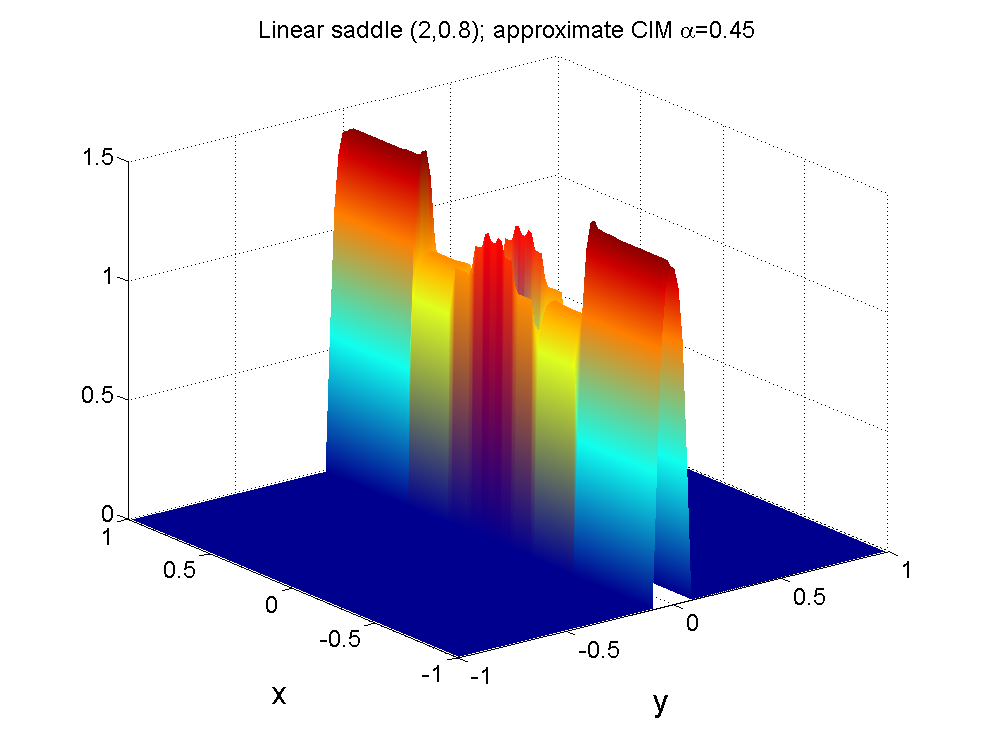}
\caption{Example~\ref{eg:saddle}. MAXENT approximations of CIMs for
for $\alpha=0.3$ (above) and $\alpha=0.45$ (below) 
for an open system with a simple saddle.}
\label{fig:saddlea}
\end{center}
\end{figure}

\begin{figure}
\begin{center}
\includegraphics[width=0.9\textwidth]{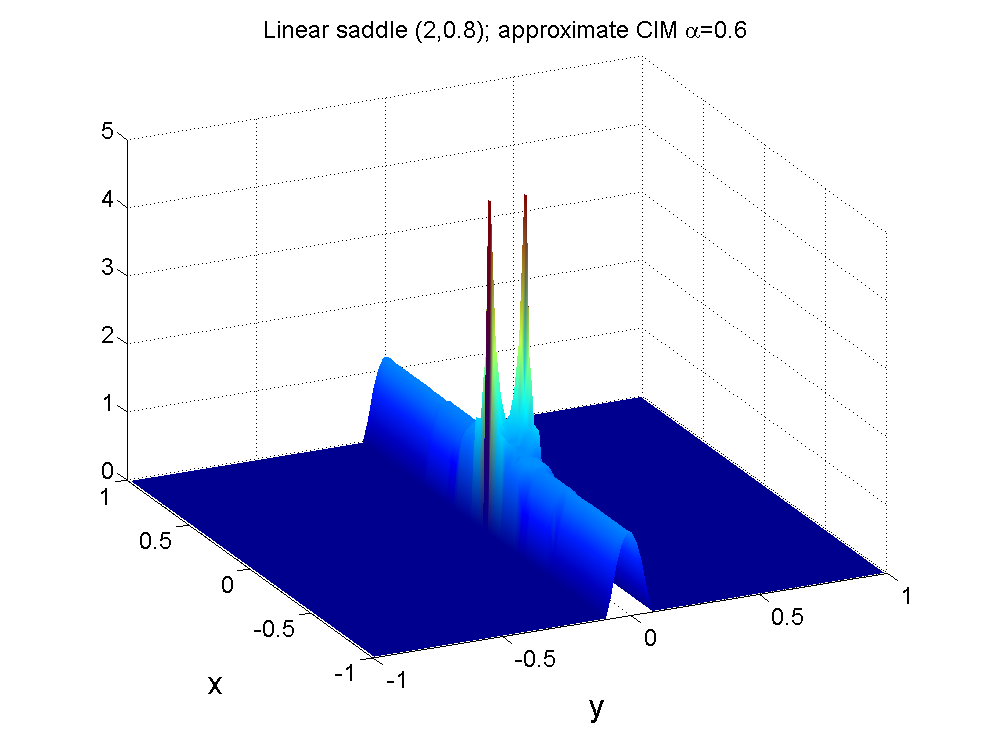}

\includegraphics[width=0.9\textwidth]{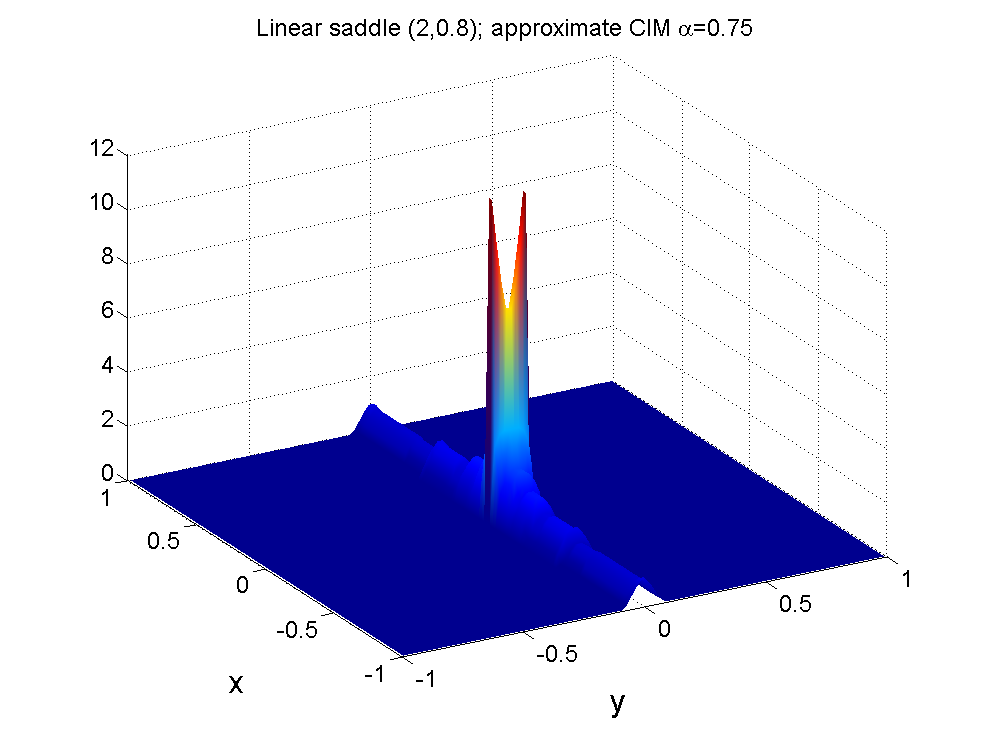}
\caption{Example~\ref{eg:saddle}. MAXENT approximations of CIMs for
for $\alpha=0.6$ (above) and $\alpha=0.75$ (below) 
for an open system with a simple saddle.}
\label{fig:saddleb}
\end{center}
\end{figure}
\end{example}

\section{Concluding remarks}

The MAXENT approach to calculating approximate ACCIMs has a sound analytical 
basis (from optimisation theory), and is easy to implement. With
test functions $\{\psi_j\}$ derived from a partition of phase space, 
the basic dynamical inputs to the computational scheme are the integrals 
$\int \psi_j\circ T\,\psi_i\,dm$ (which could be estimated from trajectory data).
For each choice of test functions, feasibility of the dual optimisation problem
depends on reducing the domain of the problem to exclude certain `backwards transient' 
parts of the phase space. With test functions derived from a partition, the resulting
`reduced domain' covers any recurrent set, and local unstable manifolds.

The work reported in this chapter suggests a number of avenues of future enquiry:
\begin{itemize}
\item are entropy-maximising ACCIMs of any particular dynamical relevance?
\item given that the analysis and computation of the variational approach
is similar with convex functionals other than $H(\cdot)$, are other choices of
objective more appropriate?
\item how is the quality of approximation affected by the choice of test 
functions $\{\psi_k\}$? 
\item how does the functional $H(f_{n,\alpha})$ depend on $\alpha$ (and $n$)?
\item can dynamically interesting measures on {\em unstable manifolds\/} be
recovered from this approach?
\end{itemize}

\bibliographystyle{spmpsci}

\end{document}